\documentclass[12pt,a4paper]{article}
\usepackage{amsmath,amssymb,latexsym,amsthm}
\usepackage[english]{babel}
\usepackage[latin2]{inputenc}

\begin{document}

\newtheorem{Dfn}{Definition}
\newtheorem{Theo}{Theorem}
\newtheorem{Lemma}[Theo]{Lemma}
\newtheorem{Prop}[Theo]{Proposition}
\newtheorem{Coro}[Theo]{Corollary}
\newcommand{\Pro}{\noindent{\em Proof. }}
\newcommand{\Rem}{\noindent{\em Remark. }}

\title{Loops as sections in compact Lie groups}
\author{\'Agota Figula and Karl Strambach}
\date{}
\maketitle
\footnotetext[1]{This paper was supported by the J\'anos Bolyai Research Fellowship and
the European Union's Seventh Framework Programme (FP7/2007-2013) under grant agreement no. 317721, no. 318202.}
\footnotetext[2]{Key words and phrases: compact connected loops, compact Lie groups, Bruschlinsky group, sharply transitive sections in compact Lie groups.}
\footnotetext[3]{2010 Mathematics Subject Classification: 22C05, 57T15, 57T20, 20N05}

\begin{abstract}
\noindent
We prove that there does not exist any connected topological proper loop homeomorphic to a quasi-simple Lie group and having a compact Lie group as the group topologically generated by its left translations. Moreover, any connected topological loop homeomorphic to the $7$-sphere and  having a compact Lie group as the group of its left translations  is classical. We give a particular simple general construction for proper loops such that the compact group of their left translations is direct product of at least $3$ factors.
\end{abstract}

\centerline{\bf 1. Introduction}

\bigskip
H. Scheerer has clarified in \cite{scheerer} for which compact connected Lie groups $G$ and for which closed subgroups $H$ the natural projection $G \to G/H$ has a continuous section $\sigma $. If $G$ is a semisimple compact Lie group, then the image $\sigma (G/H)$ is not homeomorphic to a Lie group precisely if $G$ contains a factor locally isomorphic to $PSO_8(\mathbb R)$. This is due to the fact that the group topologically generated by the left translations of the octonions of norm $1$ is the group  $SO_8(\mathbb R)$. Hence any compact connected topological loop whose group topologically generated by the left translations is a compact Lie group is itself homeomorphic to a compact Lie group

But it remained an open problem for which $\sigma$ the image $\sigma (G/H)$ determines a loop. This is the case if $\sigma (G/H)$ acts sharply transitively on $G/H$ what means that for given cosets $g_1 H$, $g_2 H$ there exists precisely one $z \in \sigma (G/H)$ such that the equation
$z g_1 H=g_2 H$ holds. Continuous sections $\sigma $ with this property (they are called sharply transitive sections) correspond to topological loops
$(L, \ast)$ (cf. \cite{loops}, Proposition 1.21, p. 29) realized on $G/H$ with respect to the multiplication
$x H \ast y H=\sigma (x H) y H$.  The group topologically generated by the left translations of $(L, \ast)$ coincides with $G$.

There are many examples of compact connected loops having a non-simple compact connected Lie group as the group topologically generated by their left translations (cf. \cite{loops}, Theorem 16.7, p. 198 and Section 14.3, pp. 170-173). A particular simple general construction for proper loops such that the group generated by their left translations is the direct product of at least three factors is given in Section $3$.

In contrast to this in this paper we prove that any connected topological loop $L$ homeomorphic to a quasi-simple Lie group $G$ and having a compact Lie group as the group topologically generated by its left translations must coincide with $G$ (cf. Theorem \ref{quasisimple}). Similarly, any connected topological loop $L$ homeomorphic to the $7$-sphere and  having a compact Lie group as the group topological generated by its left translations  is either the Moufang loop
$\mathcal O$ of octonions of norm $1$ or the factor loop $\mathcal{O}/Z$, where $Z$ is the centre of $\mathcal O$ (cf. Theorem \ref{7sphere}).

\bigskip
\noindent
\centerline{\bf 2. Prerequisites}

\bigskip
A set $L$ with a binary operation $(x,y) \mapsto x \cdot y$ is called a loop
if there exists an element $e \in L$ such that $x=e \cdot x=x \cdot e$ holds
for all $x \in L$ and the equations $a \cdot y=b$ and $x \cdot a=b$ have
precisely one solution which we denote by $y=a \backslash b$ and $x=b/a$.
The left translation $\lambda _a: y \mapsto a \cdot y :L \to L$ is a bijection of $L$ for any $a \in L$.

The kernel of a homomorphism $\alpha :(L, \circ ) \to (L', \ast )$
of a loop $L$ into a loop $L'$ is a normal subloop $N$ of $L$, i.e.
a subloop of $L$ such that
\begin{equation}  x \circ N=N \circ x, \ \ (x \circ N) \circ y=x \circ (N \circ y), \ \
x \circ (y \circ N)=(x \circ y) \circ N \nonumber \end{equation}
holds for all $x,y \in L$. A loop $(L, \cdot)$ is a product of two subloops $L_1$ and $L_2$ if any element $x$ of $L$ has a representation
$x=a \cdot b$, $a \in L_1$ and $b \in L_2$. A loop $(L, \cdot )$ is called a Moufang loop if for all $x, y, z \in L$ the identity
$(x \cdot y) \cdot (z \cdot x)=[x \cdot (y \cdot z)] \cdot x$
holds.

Let $L$ be a topological space. Then $(L, \cdot )$ is a  topological loop if  the
maps  $(x,y) \mapsto x \cdot y$, $(x,y) \mapsto x \backslash y$, $(x,y) \mapsto y / x: L^2 \to L$ are continuous. If only the multiplication and the left division are continuous, then the loop $L$  is called almost topological. An almost topological loop $L$ is a topological loop if the group generated by the left translations of $L$ is a connected Lie group (see \cite{loops},
Corollary 1.22). A loop $L$ is almost differentiable if $L$ is a differentiable manifold and the multiplication and the left division are  differentiable.

Let $G$ be a compact connected Lie group, let $H$ be a connected closed subgroup of $G$ containing no non-trivial normal subgroup of $G$ and
$G/H= \{ x H, x \in G \}$. Let $\sigma: G/H \to G$ be a continuous map with $\sigma (H)=1 \in G$ such that the set $\sigma (G/H)$ is a system of representatives for $G/H$ which generates $G$ and operates sharply transitively on $G/H$ which means that to any $x H$ and $y H$ there exists precisely one $z \in \sigma (G/H)$ with $z x H=y H$. Then the multiplication on the factor space $G/H$ given by $x H \cdot yH= \sigma(x H) yH$, respectively the multiplication on the set $\sigma (G/H)$ given by $x \cdot y=\sigma (x y H)$ yields a compact topological loop having $G$ as the group topologically generated by the left translations $x H \mapsto \sigma( x H)$, respectively $x \mapsto \sigma( x H)$.

If $L$ is a compact topological loop such that the group topologically generated by all left translations of $L$ is a compact connected Lie group, then the set $\{ \lambda _a, a \in L\}$ forms a sharply transitive section $\sigma :G/G_e \to G$ with $\sigma (\lambda _a G_e)= \lambda _a$, where $G_e$ is the stabilizer of $e \in L$ in $G$.

If the section $\sigma :G/H \to G$ is differentiable, then the loop $L$ is almost differentiable.

A  quasi-simple compact Lie group is a compact  Lie group $G$ containing a normal finite
central subgroup $N$ such that the factor group $G/N$ is simple. A semisimple connected compact
group $G$ is a Lie group containing a normal finite central subgroup $N$ such that the factor group $G/N$ is a direct product of simple Lie groups. A connected compact Lie group is an almost direct product of  compact semisimple Lie groups if its universal covering (cf. \cite{neeb}, Appendix A) is a direct product of simply connected quasi-simple Lie groups.

For connected and locally simply connected topological loops there exist universal covering loops (cf. \cite{hofmann2}, \cite{hofmann3}, \cite{hofmann}, IX.1). This yields the following lemma:

\begin{Lemma} \label{simplyconnected}
The universal covering loop $\tilde{L}$ of a connected and locally simply connected topological loop $L$ is simply connected and $L$ is isomorphic to a factor loop $\tilde{L}/N$, where $N$ is a central subgroup of $\tilde{L}$.
\end{Lemma}

\bigskip
\noindent
\centerline{\bf 3. Loops corresponding to products of groups}

\bigskip
\noindent
Let $G = K \times P \times S$ be a group, where $K$ is a group, $P$ is a non-abelian group, $g: K \to S$ is a map which is not a homomorphism such that $g(1) = 1$, the set $\{(k,1,g(k)); k \in K\}$ generates the group $K \times \{ 1 \} \times S$ and $S$ is isomorphic to a subgroup of $P$ having with the centre of $P$ trivial intersection.
Hence there is a monomorphism  $\varphi $ from $S$ into $P$ and  we may assume $H=\{(1,x,x); \ x \in S \}$. Moreover, we put
$M = \{(k, l g(k), g(k)); k \in K, \  l \in P \}$.

\noindent
Every element $(a,b,c) \in G$ may be uniquely decomposed as
$(a,b,c)=(a,b c^{-1},1)(1,c,c)$ with $(1,c,c) \in H$. Since for all $a \in K, b \in P$ there are unique elements $m=(a,b g(a),g(a)) \in M$ and $h=(1,g(a)^{-1},g(a)^{-1}) \in H$ such that $(a,b,1)=m h$ the set $M$ determines the section
$\sigma :G/H \to G; (x,y,1)H \mapsto \sigma ((x,y,1)H)=(x,y g(x),g(x))$. Since for given $a_1, a_2 \in K, b_1, b_2 \in P$  the equation
\[ (k, l g(k), g(k))(a_1,b_1,1)=(a_2,b_2,1)(1,d,d) \]
has the unique solution \[k=a_2 a_1^{-1},\  l=b_2 g(a_2 a_1^{-1}) b_1^{-1} g(a_2 a_1^{-1})^{-1} \]
with  $d=g(a_2 a_1^{-1}) \in P$, the set $M$ acts sharply
transitively on the left cosets $\{ (a,b,1) H;\ a \in K, b \in P \}$.
Since the group $H$ contains no normal subgroup of $G$ the map $\sigma $ corresponds to a  loop $L$ having the
group $G$ as the group generated by its left translations, the subgroup $H$ as the stabilizer
of $e \in L$ and the set $M$ as the set of all left translations of $L$.
The multiplication of $L$ can be defined on the set $\{ (a,b,1) H;\ a \in K, b \in P \}$ by
\begin{equation} \label{uequ1} (a_1,b_1,1) H \ast (a_2, b_2,1) H=(a_1 a_2, b_1 g(a_1) b_2 g(a_1)^{-1}, 1) H   \end{equation}
and $G = K \times P \times S$ is the group generated by the left translations of $(L, \ast )$.
Since $(1,l_1,1) H \ast (1,l_2,1) H= (1, l_1 l_2,1) H$
for all $l_1,l_2 \in P$ holds $(N, \ast )=(\{(1, l, 1) H; l \in P \}, \ast )$ is a subgroup of $(L, \ast )$ isomorphic to $P$.  
As $G$ is the direct product $G=G_1 \times G_2$ with $G_1=K \times \{ 1 \} \times S$ and $G_2=\{ 1 \} \times P \times \{ 1 \}$ and 
$\sigma (G/H)=M=M_1 \times G_2$ with $M_1=\{(k, 1, g(k)); k \in K\} \subset G_1$ it follows from Proposition 2.4 in \cite{loops}, p. 44, that the group $(N, \ast )$ is normal in the loop $(L, \ast )$.
Moreover, for all $k_1, k_2 \in K$ one has
$(k_1,1,1) H \ast (k_2,1,1) H= (k_1 k_2, 1, 1) H$.  Hence
$(K, \ast)=(\{(k,1,1) H; k \in K \}, \ast )$ is a subgroup of $(L, \ast )$ isomorphic to $K$.
Therefore the loop $(L, \ast )$ defined by (\ref{uequ1}) is a semidirect product of the normal subgroup $(N, \ast )$
by the subgroup $(K, \ast )$.

\noindent 
The loop $L$ is a group if the multiplication (\ref{uequ1}) is associative, i.e.
\begin{equation} ((a_1,b_1,1) H \ast (a_2,b_2,1) H) \ast (a_3,b_3,1) H = \nonumber \end{equation}
\begin{equation} (a_1,b_1,1) H \ast ((a_2,b_2,1) H \ast (a_3,b_3,1) H). \nonumber \end{equation}
This identity holds if and only if for all $a_i \in K$ and $b_i \in P$ one has
\begin{equation} \label{associative} (a_1 a_2 a_3, b_1 g(a_1) b_2 g(a_1)^{-1}g(a_1 a_2) b_3 g(a_1 a_2)^{-1}, 1)H= \nonumber \end{equation}
\begin{equation} (a_1 a_2 a_3,  b_1 g(a_1) b_2 g(a_2) b_3 g(a_2)^{-1} g(a_1)^{-1}, 1)H \nonumber \end{equation}
or equivalently $g(a_1)g(a_2) b_3 g(a_2)^{-1}g(a_1)^{-1} = g(a_1 a_2) b_3 g(a_1 a_2)^{-1}$.
This yields a contradiction since g is not a homomorphism. Therefore $L$ is a proper loop.

\smallskip
\noindent
If $K$ and $P$ are connected Lie groups and the
function  $g$ is  continuous, then $L$ has continuous multiplication and left division (cf. \cite{loops}, p. 29). Hence $L$ is a connected locally compact topological proper loop. If $g$ is differentiable, then $L$ is a connected almost differentiable proper loop (cf. \cite{loops}, p. 32).

\smallskip
\noindent
The constructed examples show the following

\smallskip
\noindent
\Rem
There exist proper loops with normal connected subgroups having  a compact connected Lie group $G$  as the group topologically generated by the left translations if 
$G = G_1 \times G_2 \times  G_3$, where $G_2$ is not a torus group and $G_3$ is isomorphic to a subgroup of $G_2$ having with the centre of $G_2$ trivial intersection.

\smallskip
\noindent
The aim of the paper is to demonstrate that this is a typical situation for connected compact Lie groups $G$ being groups generated by the left translations of a proper loop.

\bigskip
\centerline{\bf 4. Results}

\begin{Lemma}  \label{1dimcompact} Any one-dimensional connected topological loop having a compact Lie group as the group topologically generated by its left translations is the orthogonal group $SO_{2}(\mathbb R)$. \end{Lemma}
\begin{proof} It is proved in \cite{loops}, Proposition 18.2. \end{proof}

\begin{Lemma} \label{lemma3} Let $G$ be a connected  semisimple compact Lie group topologically generated by the left translations of a compact simply connected loop which is homeomorphic to a semisimple Lie group $K_1$. Let $H$ be the stabilizer of $e \in L$. Then one has
$H=(H_{1}, \rho (H_{1}))=\{ (x, \rho(x)), x \in H_1 \}$ and $G= K_{1} \times  \rho (H_{1})$, where $\rho$ is a monomorphism and $H_1$ is a subgroup of $K_1$. Moreover, $H_1$ has with the centre $Z_1$ of $K_1$ a trivial intersection.
\end{Lemma}
\begin{proof} Since $G/H$ is homeomorphic to $K_1$ the group $G$ is homeomorphic to $K_1 \times H$. According to \cite{scheerer} or to Theorem 16.1 in \cite{loops}, p. 195, the group $G$ has the form $G=K_1 \times K_2$ and the stabilizer $H$ of $e \in L$ is $H=(H_{1}, \rho (H_{1}))$, where $K_2$ is a Lie group isomorphic to $H$ and 
$\rho $ is a monomorphism. From this it follows that $G=K_1 \times \rho (H_{1})$.
If $Z_1 \cap H_1 \neq 1$, then $H$ has  with the centre $Z$ of $G$ a non-trivial intersection. But this is a contradiction to the fact that
 $Z \cap H = 1$.
\end{proof}

\smallskip
\noindent
Lemma \ref{lemma3} yields

\begin{Coro} \label{coro1} Let $G=K_1 \times K_2$ be a compact semisimple Lie group such that $K_1$ is semisimple and let $H$ be a subgroup of $G$ such that $H=(H_1, \rho (H_{1}))$, where $H_1$ is a subgroup of $K_1$ and $\rho $ is a monomorphism. If $K_1$ has a non-trivial centre $Z_1$ and
$H_1 \cap Z_1 \neq 1$, then there exists no proper loop $L$ homeomorphic to $K_1$ such that $G$ is the group topologically generated by the left translations of $L$ and $H$ is the stabilizer of $e \in L$.
\end{Coro}

\begin{Coro} \label{coro2}
There does not exist a connected topological proper loop $L$ homeomorphic to a covering of a product $K_1$ of the groups 
$SO_3(\mathbb R)$ and 
having a compact semisimple Lie group
$G=K_1 \times K_2$ as the group  topologically generated by the left translations of $L$.
\end{Coro}
\begin{proof}
We may assume that $L$ is simply connected and hence $K_1$ is a direct product of groups isomorphic to $Spin_3(\mathbb R)$.
Then the stabilizer $H$ of $e \in L$ is the subgroup $H=(H_1, \rho (H_{1}))$, where $H_1$ is a subgroup of $K_1$ and $\rho $ is a monomorphism.
The group $G$ topologically generated by the left translations of $L$ has the form
$G=K_1 \times \rho (H_{1})$ (cf. Lemma \ref{lemma3}). The assertion follows from Corollary \ref{coro1} because any subgroup $H_1$ intersects the centre of $K_1$ non trivially.
\end{proof}

\begin{Theo} \label{7sphere} Let $L$ be a topological loop  homeomorphic to the $7$-sphere or to the $7$-dimensional real projective space such that the group $G$ topologically generated by the left translations of $L$ is a compact Lie group. Then $L$ is one of the two $7$-dimensional compact Moufang loops, $G$ is locally isomorphic to $PSO_{8}(\mathbb R)$ and the stabilizer $H$ of $e \in L$ is isomorphic to $SO_{7}(\mathbb R)$.
\end{Theo}
\begin{proof}
We may assume that $L$ is simply connected. Since $G$ is a compact Lie group using Proposition 2.4 in  \cite{hudson} and Ascoli's Theorem,  from IX.2.9 Theorem of \cite{hofmann} it follows that the loop $L$ has a left invariant uniformity. Therefore IX.3.14 Theorem in \cite{hofmann} yields that $L$ is the multiplicative loop $\cal O$ of octonions having norm $1$.
Then $G$ is isomorphic to
$SO_{8}(\mathbb R)$ and the stabilizer $H$ of $e \in L$ is isomorphic to $SO_{7}(\mathbb R)$ (cf. \cite{scheerer}).

If $L$ is homeomorphic to the $7$-dimensional real projective space, then the universal covering $\tilde{L}$ of $L$ is a Moufang loop homeomorphic
to the sphere $S^7$. It follows from \cite{hofmann}, p. 216, that the loop $L$ is a factor loop $\tilde{L}/N$, where $N$ is a central subgroup  of $\tilde{L}$ of order 2. Lemma 1.33 in \cite{loops} yields that $L$ is the Moufang loop ${\cal O} /Z$, where $Z$ is the centre of the multiplicative loop of octonions having norm $1$.
\end{proof}

\smallskip
\noindent
If $L$ is a topological loop  homeomorphic to the $7$-sphere and if we assume that the group $G$ topologically generated by the left translations of $L$ is a quasi-simple compact Lie group, then
$G$ is isomorphic to $SO_8(\mathbb R)$ and the stabilizer $H$ of $e \in L$ is isomorphic to $SO_7(\mathbb R)$. This allows us to obtain the assertion of the previous theorem also in the following way.
We identify the set $G/H$ of the left cosets with the set $\cal S$ of the left translations of
the loop $\cal O$. The section
$\sigma: G/H \rightarrow G$ belonging to a topological loop $L$ has the form
$\sigma(x H) = x \phi (x)$, where $x \in \cal S$ and $\phi $ is a continuous map from $\cal S$ to $H$. Since any two elements of $\cal S$  are contained in a subgroup $D$ isomorphic to $Spin_{3}(\mathbb R)$ the restriction of $\phi $  to $D$ is a homomorphism (Corollary \ref{coro2}). Hence $L$ is a  diassociative Lie loop (\cite{hofmann},  IX.6.42)  and  Theorem 16.10 in \cite{loops} yields the assertion of the previous theorem.

\medskip

\begin{Theo} \label{theo1} Let $G$ be a compact Lie group which is the group topologically generated by the left
translations  of a proper topological loop $L$ homeomorphic to a connected semisimple compact Lie group. Then $G$ is a connected semisimple Lie group.
\end{Theo}
\begin{proof} Since $L$ is connected also $G$ is connected. By Hofmann-Scheerer Splitting Theorem 
(cf. \cite{neeb}, p. 474) the group $G$ is isomorphic to a semidirect product $G = G' \rtimes T$, where $G'$ is the semisimple commutator subgroup of $G$  and $T$ is a torus. The group $G'$ is isomorphic to an almost direct product 
$G' = K_{1} \cdots K_m$ of quasi-simple compact Lie  groups.
The loop $L$ is homeomorphic to a connected semisimple compact Lie group $K=K_1 \cdots K_s$ with $s \le m$.
We may assume that $L$ and hence also $K$ is simply connected. Since the universal covering ${\tilde G'}$ of $G'$ is the direct product of $K$ and the universal covering ${\tilde S}$ of $S=K_{s+1} \cdots K_m$, the group $G'$ is the direct product of $K$ and $S$. As $L$ is homeomorphic to the image of the section $\sigma: G/H \rightarrow G$, where $H$ is the stabilizer of $e \in L$, the set $\sigma(G/H)$  has the form $\{(x, \alpha(x)) \}$, where $x \in K$ and $\alpha $ is a continuous mapping from $K$ into
$S \rtimes T$. The group $T = T_1\times \cdots \times T_h$ is the direct product of one-dimensional tori $T_i$. Let $\pi$ be the projection from $S\rtimes T$ into $T$ along $S$ and 
$\iota_i$ be the projection from $T$ into $T_i$ along the complement 
$\prod_{j\neq i} T_j$. As  $\sigma(G/H)$ is a compact connected homogeneous space and $T_i$ is a $1$-sphere for all $i$ any   $\iota_i \pi \alpha (K)$ is either constant or surjective. Since 
$\sigma(G/H)$ generates $G$ there exists one $i$ such that $\iota_i \pi \alpha (K)$ is different from $\{1 \}$. As the group $T$ is the direct product of $1$-dimensional tori $T_i$  the Bruschlinsky group $B$ (cf. \cite{hu}, p. 47) of $K$ is not trivial. By Theorem 7.1 in \cite{hu}, p. 49, $B$ is isomorphic to the first cohomology group $H^{1}(K)$. The graded cohomology algebra of the compact Lie group $K$ is the tensor product of the cohomology algebras $H^{1}(F_i)$ of the quasi-simple factors $F_i$ of $K$. Since $H^{1}(F_i)$
has no generators of degree $1$ and $2$ (\cite{gorbatsevich}, pp. 126-127) also the cohomology algebra $H^{1}(K)$ has no generators of degree $1$ and $2$. Hence the Poincare polynomial $\psi (K)$ has no linear and quadratic monomials, which is a contradiction.
\end{proof}

\medskip
\noindent
\Rem  In contrast to the previous theorem a non semisimple compact Lie group may be the  group topologically generated by the left translations of a loop $L$  if $L$ is homeomorphic to a  non semisimple compact connected Lie group.\\
Let $T_1$ be a torus of dimension $m\geq 1$, let $P$ be a connected semisimple compact Lie group and let $T_2$ be a torus of dimension $s$ with 
$1\leq s \leq m$ such that there exists a monomorphism 
$\varphi :T_2 \to P$ with $\varphi (T_2) \cap Z(P)=\{1 \}$, where $Z(P)$ is the centre of $P$. If $g:T_1 \to T_2$ is a  continuous surjective mapping with $g(1)= 1$ which is not a homomorphism, then with the subgroup 
$H=\{(1,\varphi (x),x); x \in T_2\}$ of $G = T_1 \times P \times T_2$ as the stabilizer there exists according to Section 3 a proper connected loop $L$ homeomorphic to $T_1 \times P$ having the direct product $T_1 \times P \times T_2$ as the group topologically generated by the left translations of $L$.

\begin{Theo} \label{quasisimple} There does not exist any proper topological loop which is homeomorphic to a connected quasi-simple Lie group and has a compact Lie group as the group topologically generated by its left translations.
\end{Theo}
\begin{proof}
By Lemma \ref{simplyconnected}  we may assume that $L$ is a proper loop  homeomorphic to a simply connected quasi-simple compact Lie group $K_1$. Then the stabilizer $H$ of $e \in L$ has the form
$H = (H_{1},\rho (H_{1}))=\{ (x, \rho(x)); x \in H_1\}$, where $\rho$ is a monomorphism and the group $G$ topologically generated by the left translations of
$L$ has the form
$G = K_{1} \times  \rho (H_{1})$ (cf. Lemma \ref{lemma3}). Identifying the space $G/H$ with $K_1$ one has that the image
$\sigma (K_1)$ of the section $\sigma: K_1 \rightarrow G$ intersects  $H$ trivially.
As $\rho$ is a monomorphism we may assume that $H = (H_1, H_1) = \{(x, x); x \in H_1\}$.
The restriction of  $\sigma $ to a one-dimensional torus subgroup $A$
of $H_1$ yields $\sigma (A)  = \{(u, f(u))\}$, where $f$ is a continuous function. Since the compact loop $\sigma(A)$ is a group
(cf. Lemma \ref{1dimcompact}) the map $f$ is a homomorphism. It follows that $\sigma (A)$ has the form  $\{(u, u^n)\}$ with fixed
$n\in \mathbb Z$. Since $\{(u, u^n); u \in A \cong SO_2(\mathbb R)\} \cap H = \{ 1 \}$ the equation $x^n=x$, $x \in A$, can be satisfied only for $x=1$. Equivalently, $x \mapsto x^{n-1}$ is an  automorphism of $SO_2(\mathbb R)$. Besides the identity the only non-trivial automorphism of the group $SO_2(\mathbb R)$ is the map
$x \mapsto x^{-1}$. Therefore we get $n\in \{0, 2 \}$.
Let $C_1$ be a $3$-dimensional  subgroup of $H_1$. Since any $3$-dimensional compact loop which has a compact Lie group as the group topologically generated by its left translations is a group (cf. Corollary \ref{coro2}) 
$C=\sigma (C_1)= (C_1, \psi (C_1))$ is locally isomorphic to $SO_3(\mathbb R)$ and $\psi $ is a homomorphism of $C_1$. Besides a homomorphism with finite kernel any continuous homomorphism is an automorphism induced by a conjugation with elements of the orthogonal group $O_{3}(\mathbb R)$. Hence for no $1$-dimensional subgroup $A$ of 
$C_1$ one can have $\sigma (A)=\{(x, x^2), x \in A \}$. As in compact groups the exponential map is surjective the compact group $C$ is the union of the one-dimensional connected subgroups $\sigma (A)=\{(x,1), x \in A \}$. Hence $C$ has the form $(C_1,1)$. Since the $3$-dimensional subgroups of $H_1$ covers $H_1$ 
(cf. \cite{hofmann4}, Propositions 6.45 and 6.46) for the continuous section $\sigma $ one has 
$\sigma (H_1)=(H_1,1)$. 

Let $B_i$ be a one-dimensional torus subgroup of $K_1$ such that $\sigma (B_i) = (B_i, 1)$. The union $B =\bigcup B_i$ of the one-dimensional subgroups of $K_1$ forms a subgroup of 
$K_1$ containing $H_1$.

Let $F_i$ be a $1$-dimensional torus subgroup of $K_1$ such that $\sigma (F_i) \neq (F_i,1)$. Then one has $\sigma (F_i)=\{ (x, x^n); x \in F_i \}$, where 
$n \in \mathbb Z \setminus \{ 0 \}$. Since any $1$-dimensional subgroup of $K_1$ is contained in a $3$-dimensional subgroup of $K_1$ locally isomorphic to $SO_3(\mathbb R)$ 
(cf. \cite{hofmann4}, Propositions 6.45 and 6.46) and by Corollary \ref{coro2} any $3$-dimensional loop homeomorphic to a cover of $SO_3(\mathbb R)$ is a group, besides a homomorphism with finite kernel we get that 
$x \mapsto x^n$ is either an isomorphism or an anti-isomorphism of $SO_3(\mathbb R)$.  Hence one has $n=1$ or $-1$. It follows that either
$\sigma (F_i)=\{ (x, x); x \in F_i \}$ or $\sigma (F_i)=\{ (x, x^{-1}); x \in F_i \}$ for any $1$-dimensional torus subgroup $F_i$ of $K_1$ such that $\sigma (F_i) \neq (F_i,1)$. The union $F= \cup F_i$ of the $1$-dimensional torus subgroups $F_i$ of $K_1$ with $\sigma (F_i) \neq (F_i,1)$ is isomorphic to the group $\rho (H_1) \cong H_1$. 

The subgroups $F$ and $B$ yield a factorization of $K_1$ such that the intersection $F \cap B$ is discrete  which is a contradiction to the fact that $K_1$ is quasi-simple
(cf. Theorem 4.6 in \cite{gorbatsevich}, p. 145).
 \end{proof}

\begin{Coro} Let $L$ be a proper topological loop homeomorphic to a product of quasi-simple simply connected compact  Lie groups and having a compact Lie group $G$ as the group topologically generated by its left translations. Then $G$ is at least $14$-dimensional.\\
If $dim G = 14$, then $G$ is locally isomorphic to $Spin_3(\mathbb R) \times SU_3(\mathbb C) \times Spin_3(\mathbb R)$ and $L$ is homeomorphic to a group which is locally isomorphic to $Spin_3(\mathbb R) \times SU_3(\mathbb C).$
\end{Coro}
\begin{proof} We assume that the loop $L$ is simply connected. Then $L$ is homeomorphic to the direct product $K_1$ of at least two quasi-simple simply connected factors (cf. Theorem \ref{quasisimple}). According to Theorem \ref{theo1} the connected group $G$ is semisimple. Hence by
Lemma \ref{lemma3} the stabilizer $H$ of $e \in L$ has the form $(H_1, \rho(H_1))=\{ (x, \rho(x)), x \in H_1\}$ and
$G=K_1 \times \rho(H_1)$, where $H_1$ is a subgroup of $K_1$ and $\rho $ is a monomorphism. Since any subgroup of 
$Spin_3(\mathbb R)$ intersects its centre not trivially according to Corollary  \ref{coro2} and to the construction in Section 3 the group $K_1$ coincides with 
$K \times P=Spin_3(\mathbb R) \times SU_3(\mathbb C)$, the subgroup $H_1$ has the form $1 \times Spin_3(\mathbb R)$ and 
$\rho : H_1 \to S$ is an isomorphism. Therefore one has 
$S=Spin_3(\mathbb R)$ and for the function $g: K \to S$ one can choose the function $x \mapsto x^{n}$ with 
$n \in \mathbb Z \setminus \{-1, 0, 1\}$. Hence from the construction in Section 3 we have 
$G=Spin_3(\mathbb R) \times SU_3(\mathbb C) \times Spin_3(\mathbb R)$.
\end{proof}

\medskip
\noindent
\Rem  Euclidean and hyperbolic symmetric spaces correspond to global differentiable loops 
(cf. \cite{loops}, Theorem 11.8, p. 135). In contrast to this, compact simple symmetric spaces which are not Lie groups yield only local Bol loops $L$ since for $L$ the exponential map is not a diffeomorphism 
(cf. \cite{loops}, Proposition 9.19, p. 115).

\bigskip
\noindent
\'Agota Figula, Institute of Mathematics, University of Debrecen,
\newline
H-4010 Debrecen, P.O.B. 12, Hungary, figula@science.unideb.hu

\smallskip
\noindent
Karl Strambach, Universit\"at Erlangen-Nürnberg, Department Mathematik,
\newline
Cauerstrasse 11, D-91058 Erlangen, Germany, stramba@math.fau.de

\end{document}